\documentclass[a4paper,UKenglish,cleveref,numberwithinsect,autoref]{lipics-v2021}

\pdfoutput=1 
\hideLIPIcs  

\usepackage[T1]{fontenc}
\usepackage[utf8]{inputenc}

\usepackage[scaled=.85]{beramono} 

\usepackage{amsmath}
\usepackage{amsfonts}
\usepackage{amssymb}
\usepackage{bbm}
\usepackage{upgreek} 

\usepackage[scr=rsfso,cal=boondoxo]{mathalfa} 

\usepackage{xypic}
\usepackage{xcolor}
\usepackage[shortlabels]{enumitem}

\usepackage{thmtools} 

\numberwithin{equation}{section}

\newcommand{\defemph}[1]{\textbf{\emph{#1}}} 

\newcommand{\defeq}{\mathchoice%
{\;\mathrel{{:}{=}}\;}
{\mathrel{{:}{=}}}
{\mathrel{{:}{=}}}
{\mathrel{{:}{=}}}
}

\newcommand{\defiff}{\mathchoice%
{\;\mathrel{{:}{\Leftrightarrow}}\;}
{\mathrel{{:}{\Leftrightarrow}}}
{\mathrel{{:}{\Leftrightarrow}}}
{\mathrel{{:}{\Leftrightarrow}}}
}

\renewcommand{\j}{\mathcal{j}} 
\renewcommand{\k}{\mathcal{k}} 

\newcommand{\Prop}{\mathsf{Prop}} 
\newcommand{\DProp}{\mathsf{DProp}} 
\newcommand{\oracle}[1]{\mathcal{o}_{#1}} 
\newcommand{\open}[1]{\bigcirc_{#1}} 

\newcommand{\predicate}[1]{\mathcal{p}_{#1}} 

\newcommand{\Type}{\mathsf{Type}} 

\newcommand{\lthen}{\Rightarrow} 
\newcommand{\liff}{\Leftrightarrow} 

\newcommand{\all}[1]{\forall #1.\,} 
\newcommand{\some}[1]{\exists #1.\,} 
\newcommand{\someone}[1]{\exists! #1.\,} 
\newcommand{\dprod}[1]{\Uppi{(#1)}.\,} 
\newcommand{\dsum}[1]{\Upsigma{(#1)}.\,} 
\newcommand{\lam}[1]{\uplambda{#1}.\,} 

\newcommand{\vthin}{\mkern1.5mu} 
\newcommand{\ap}{\vthin} 
\newcommand{\of}{\vthin{:}\vthin}

\newcommand{\one}{\mathbbm{1}} 
\newcommand{\two}{\mathbbm{2}} 
\newcommand{\NN}{\mathbbm{N}} 
\newcommand{\lpo}{\mathsf{lpo}} 
\newcommand{\lem}{\mathsf{lem}} 

\newcommand{\set}[1]{\{#1\}} 
\newcommand{\such}{\mid}

\newcommand{\bind}{\mathsf{bind}} 

\newcommand{\cont}[3][]{#2 \mathbin{\lhd}_{#1} #3} 

\newcommand{\Cont}{\mathsf{Cont}} 
\newcommand{\PCont}{\mathsf{PCont}} 

\newcommand{\Mod}{\mathsf{Mod}} 

\newcommand{\Tree}[1][]{\mathcal{T}_{#1}} 
\newcommand{\EquiTree}[1][]{\mathcal{E}_{#1}} 
\newcommand{\equi}{\mathsf{equi}} 

\newcommand{\leaf}{\mathsf{leaf}} 
\newcommand{\node}{\mathsf{node}} 

\newcommand{\prf}{\mathsf{prf}} 
\newcommand{\ask}{\mathsf{ask}} 

\newcommand{\ret}{\mathsf{ret}} 

\newcommand{\ileq}{\leq_{\mathsf{I}}} 

\newcommand{\down}[1]{{\downarrow}#1} 
\newcommand{\pt}{\mathsf{pt}} 

\newcommand{\sheafify}[2][]{\mathsf{shfy}_{#1}#2} 

\renewcommand{\AA}{\mathbbm{A}} 
\newcommand{\RT}[1]{\mathsf{RT}(#1)} 
\newcommand{\pair}[1]{\langle #1 \rangle} 
\newcommand{\numeral}[1]{\overline{#1}} 
\newcommand{\pow}[1]{\mathcal{P} #1} 

\newcommand{\asm}[1]{\mathbf{#1}} 
\newcommand{\rz}[1][]{\Vdash_{#1}} 
\newcommand{\app}{\cdot} 
\newcommand{\defined}[1]{#1{\downarrow}} 
\newcommand{\Oracle}[1]{\mathcal{O}_{#1}} 
\newcommand{\wleq}{\leq_{\mathsf{W}}} 
\newcommand{\support}[1]{\|#1\|} 

\newcommand{\Trunc}[1]{\lVert #1\rVert} 

\title{Sheaves as oracle computations} 

\author{Danel Ahman}{Institute of Computer Science, University of Tartu, Estonia \and \url{https://danel.ahman.ee}}{danel.ahman@ut.ee}{0000-0001-6595-2756}{This work was supported by the Estonian Research Council grant PRG2764.}

\author{Andrej Bauer}{Faculty of Mathematics and Physics, University of Ljubljana, Slovenia \and Institute of Mathematics, Physics and Mechanics, Slovenia \and \url{https://www.andrej.com}}{Andrej.Bauer@andrej.com}{0000-0001-5378-0547}{This material is based upon work supported by the Air Force Office of Scientific Research under award number FA9550-21-1-0024.
}

\authorrunning{D.~Ahman and A. Bauer} 

\Copyright{Danel Ahman and Andrej Bauer} 



\supplement{}
\supplementdetails[subcategory={Formalization}]{Software}{https://doi.org/10.5281/zenodo.19858008}

\acknowledgements{
  We thank Takayuki Kihara, Andrew Swan, and Sewon Park for enlightening discussions,
  and Martin Baillon, Assia Mahboubi, and Pierre-Marie Pédrot for a confluence of ideas.
}

\nolinenumbers 

\EventEditors{Frank Pfenning}
\EventNoEds{1}
\EventLongTitle{11th International Conference on Formal Structures for Computation and Deduction (FSCD 2026)}
\EventShortTitle{FSCD 2026}
\EventAcronym{FSCD}
\EventYear{2026}
\EventDate{July 20--23, 2026}
\EventLocation{Lisbon, Portugal}
\EventLogo{}
\SeriesVolume{378}
\ArticleNo{33}

\begin{document}

\maketitle

\begin{abstract}
In type theory, an oracle may be specified abstractly by a predicate whose domain is the type of queries asked of the oracle, and whose proofs are the oracle answers. Such a specification induces an oracle modality that captures a computational intuition about oracles: at each step of reasoning we either know the result, or we ask the oracle a query and proceed upon receiving an answer.
We characterize an oracle modality as the least one forcing the given predicate. We establish an adjoint retraction between modalities and propositional containers, from which it follows that every modality is an oracle modality. The left adjoint maps sums to suprema, which makes suprema of modalities easy to compute when they are given in terms of oracle modalities.
We also study sheaves for oracle modalities. We describe sheafification in terms of a quotient-inductive type of computation trees, and describe sheaves as algebras for the corresponding monad. We also introduce equifoliate trees, an intensional notion of oracle computation given by a (non-propositional) container. Equifoliate trees descend to sheaves, and modally cover them. As an application, we give a concrete description of all Lawvere-Tierney topologies in a realizability topos, closely related to a game-theoretic characterization by Takayuki Kihara.
\end{abstract}

\section{Introduction}
\label{sec:introduction}

Sheaves are a fundamental notion in several branches of mathematics, notably in geometry and algebra, as well as in logic. They are less frequently pursued from the computational point of view, which is the topic of our paper.

Our starting point is the definition of an \emph{oracle modality}, which captures the intuition that reasoning relative to an oracle proceeds by asking the oracle to provide evidence of statements that we cannot prove constructively, for instance, the decidability of the halting of a given Turing machine.
In \cref{sec:oracle-modalities} we provide inductive and fixed-point constructions of oracle modalities, and show that they are characterized as the least modalities forcing the specifying predicate.
After relating oracle modalities to instance reducibility in \cref{sec:open-modalities,sec:oracle-modal-inst}, we show in \cref{sec:every-modality-oracle} that the oracle modality construction is a functor from the category of propositional containers to the frame of modalities. Not only that, the functor has a section, which is also its right adjoint, from which it follows that every Lawvere-Tierney topology \cite{lawvere70:_quant} is in fact an oracle modality.

Encouraged by the generality of our approach, in \cref{sec:sheaves} we pass from modal proofs to modal constructions.
An element of a type $X$ constructed relative to an oracle is an element of the \emph{sheafification} of~$X$, which we describe as a quotient-inductive type. The inductive part provides a computational view of constructions relative to an oracle, where again at each step we either know an element of~$X$ or we ask the oracle a question. However, this view is of limited value because the quotient masks intensional information about the computation, such as which queries are asked in the construction.
We thus provide in~\cref{sec:oracle-modal-pred} a second description of sheafification in terms of \emph{equifoliate trees}. They capture intensional information one normally expects when studying effectful computations, because they are ordinary, unquotiented inductive trees built from an arbitrary, possibly non-propositional container, constrained to compute at most one element up to modal equality.
We prove that equifoliate trees form a monad, they descend to the sheafification for the associated propositionally truncated container, and that every element of the sheafification arises, modally, from an equifoliate tree.

As an application, in \cref{sec:modalities-realizability} we characterize all Lawvere-Tierney topologies in a relative realizability topos in terms of oracle modalities.
We give two explicit descriptions, one in terms of realizers encoding oracle computation for a predicate on a partitioned assembly, and another for an extended Weihrauch predicate.
The description is closely related to a recent game-theoretic treatment of Lawvere-Tierney topologies by Takayuki Kihara~\cite{kihara23:_lawver_tiern}.


\section{Type-theoretic background}
\label{sec:background}

For the majority of the paper, we will work in what is essentially the
underlying type theory of Rocq, i.e., the Calculus of Inductive Constructions
(CIC)~\cite{paulinmohring:cic-intro}. 
Specifically, we have dependent products $\dprod{x \of A} B(x)$, dependent sums $\dsum{x \of A} B(x)$, the identity type, which we write as $a =b$, as well as inductive types with dependent elimination. For instance, we respectively write $\two$ and $\NN$ for the standard inductive types of booleans
and natural numbers.
CIC has an infinite cumulative hierarchy of (predicative) type universes $\Type_0 : \Type_1 : \cdots$
which does not play a central role in our development. We just work at any chosen level $\Type_i$, which we simply write as $\Type$.

In several places we rely crucially on having an \emph{impredicative} universe $\Prop$ of propositions, i.e., one that allows quantification over arbitrary types.
For better readability, we use logical notation $\forall$, $\exists$, $\lthen$, $\liff$ for constructions in~$\Prop$
and the type-theoretic one $\Uppi$, $\Upsigma$, $\to$, $\leftrightarrow$ for operations on types.
Note that $\Prop$ is proof-irrelevant, i.e., from propositions we may eliminate only to propositions, and not to general types.
In particular, from a proof of $\some{x \of A} P \ap x$ we cannot in general extract an $a : A$ such that $P \ap a$.
Additionally, we assume the following principles, which are valid in toposes:
\begin{itemize}
\item \defemph{function extensionality}: functions giving equal values are equal,
\item \defemph{propositional extensionality}: if $p, q : \Prop$ and $p \liff q$, then $p = q$,
\item \defemph{uniqueness of proofs}: if $p : \Prop$ and $u, v : p$, then $u = v$,
\item \defemph{equality is propositional}: the codomain of the identity type is~$\Prop$.
\item \defemph{definite description}: if $\someone{x \of A} P \ap x$ then we have $a : A$ such that $P \ap a$.
\end{itemize}
Here $\someone{x \of A} P \ap x$ is the usual shorthand for $\some{x \of A} (P \ap x \land \all{y \of A} (P \ap y \lthen x = y))$ which expresses unique existence.

Propositionality of equality, together with uniqueness of proofs, implies \defemph{uniqueness of identity proofs} (UIP), which states that any two proofs of an equality are equal.
Both are used throughout \cref{sec:sheaves}, see \cref{prop:sheaf-tfae,thm:oracle-sheaf-when} for representative use.

A consequence of definite description is \defemph{unique choice}: if $P : A \to B \to \Prop$ and $\all{x \of A} \someone{y \of B} P \ap x \ap y$ then we have $f : A \to B$ such that $\all{x \of A} P \ap x \ap (f \ap x)$.
Any further extensions of type theory, such as quotient-inductive types used in \cref{sec:sheaves}, are discussed explicitly when used.

Much of our development can be best framed in terms of containers~\cite{AbbottAG:Containers}, also known as univariate polynomials~\cite{GambinoK:PolyMonads,PolyBook}. Recall that a \defemph{container}
$\cont A P$ comprises a type~$A$ of \defemph{shapes} and a type family~$P : A
\to \Type$ of \defemph{positions}.
Originally developed to model parametric data types and polymorphic functions, over time containers have found applications in many other areas of computer science and mathematics. 
They form a category $\Cont$, with a morphism $\cont{t}{q} : \cont{A}{P} \to \cont{B}{Q}$ given by maps $t : A
\to B$ and $q : \dprod{a \of A} Q (t \ap a) \to P \ap a$.
We shall mainly work with \defemph{propositional containers}, which are containers whose families of positions are
predicates $P : A \to \Prop$.
Propositional containers form a full subcategory $\PCont$ of $\Cont$.
We usually refer to a propositional container $\cont{A}{P}$ just by its predicate $P$.






\section{Oracle modalities}
\label{sec:oracle-modalities}

We take a \defemph{modality} to be a monad on $\Prop$ qua category, i.e., a map $\j : \Prop \to \Prop$ that is
\begin{itemize}
\item monotone: $\all{(p \ap q \of \Prop)} (p \lthen q) \lthen (\j \ap p \lthen \j \ap q)$,
\item inflationary: $\all{p \of \Prop} p \lthen \j \ap p$, and
\item idempotent: $\all{p \of \Prop} \j \ap (\j \ap p) \lthen \j \ap p$, 
\end{itemize}
also known as
a \emph{Lawvere-Tierney topology}~\cite{lane92:_sheav_geomet_logic},
a \emph{local operator}~\cite{oosten08:_realiz}, 
and a \emph{nucleus}~\cite{johnstone82:_stone}.

Every modality preserves conjunctions, $\all{(p, q \of \Prop)} \j (p \land q) = \j p \land \j q$.
Modalities form a frame~$\Mod$~\cite[\S2.5]{johnstone82:_stone} with respect to the pointwise ordering
$
  \j \leq \k \defiff \all{p \of \Prop} (\j \ap p \lthen \k \ap p)
$.

An \defemph{(abstract) oracle} is specified by a predicate $P : A \to \Prop$, where we think of an element $a : A$ as a \defemph{query} which the oracle \defemph{answers} by providing evidence of $P \ap a$. 
In the interesting cases the oracle so specified cannot be shown to exist constructively, i.e., $\all{a \of A} P \ap a$ is not provable. Nevertheless, we would still like to reason as if it were.

\begin{example}
  \label{exa:oracle-lem}%
  The excluded middle oracle $\lem : \Prop \to \Prop$ is given by $\lem \ap p \defeq p \lor \neg p$.
\end{example}

\begin{example}
  \label{exa:oracle-lpo}%
  The \emph{Limited Principle of Omniscience (LPO)}~\cite{bishop67} states that every infinite binary sequence either contains~$1$ or is constantly~$0$. The oracle corresponding to it is specified by the predicate $\lpo : \two^\NN \to \Prop$, defined as $\lpo \ap \alpha \defeq (\some{n} \alpha \ap n = 1) \lor (\all{n} \alpha \ap n = 0)$.
  Other instances of excluded middle can be devised in a similar fashion.
\end{example}

\begin{example}
  \label{exa:traditional-oracle}%
  A traditional oracle in computability theory decides membership in a subset $S \subseteq \NN$. In our setting,
  such a subset is expressed as a predicate $S : \NN \to \Prop$, which induces the oracle $\mathsf{mem}_S : \NN \to \Prop$ defined by $\mathsf{mem}_S \ap p \defeq S \ap p \lor \neg S \ap p$.
\end{example}

\begin{example}
  \label{exa:oracle-realized}%
  If $P : A \to \Prop$ specifies an oracle which is implemented by $f : \all{a \of A} P \ap a$, then no additional computational power is gained by having an oracle for~$P$, as we can replace it with~$f$.
  In \cref{exa:traditional-oracle} this happens when~$S$ is decidable.
\end{example}

\begin{example}
  \label{exa:oracle-counterexample}%
  The other trivial case arises when $P : A \to \Prop$ has a counterexample $a : A$ so that $\neg P \ap a$. In this case, the oracle gets stuck on the query~$a$ because~$P \ap a$ has no answer.
\end{example}

Intuitively speaking, to show that~$s : \Prop$ holds relative to an oracle specified by~$P : A \to \Prop$ we may either establish~$s$ without consulting the oracle, or ask the oracle a query~$a : A$ and, upon receiving evidence $u : P \ap a$, proceed recursively. Eventually we must stop querying the oracle and resolve~$s$ with the accumulated knowledge about~$P$.
The intuition is captured by the \defemph{oracle modality~$\oracle{P} : \Prop \to \Prop$}, defined at $s : \Prop$ inductively by the clauses
\begin{itemize}
\item $\prf : s \to \oracle{P} \ap s$,
\item $\ask : \dprod{a \of A} (P \ap a \to \oracle{P} \ap s) \to \oracle{P} \ap s$.
\end{itemize}
Thus, an element of $\oracle{P} \ap s$ is either $\prf \ap u$, where~$u$ is a direct proof of~$s$,
or $\ask \ap a \ap \kappa$, where~$a : A$ is a query and $\kappa : P \ap a \to \oracle{P} \ap s$ the continuation of the proof of $\oracle{P} \ap s$.
Beware, one should not think of $\oracle{P} \ap s$ as ``$s$ follows from finitely many instances of~$P$'', because the instances appearing in a proof of~$\oracle{P} \ap s$ cannot in general be collected without the aid of an oracle,
for access to the next query is guarded by an answer to the previous one.

Up to equivalence, $\oracle{P} \ap s$ may be constructed as the least fixed-point of the monotone map $f^P_s : \Prop \to \Prop$ defined by
\begin{equation}
  \label{eq:oracle-map}
  f^P_s \ap t \defeq s \lor \some{a \of A} (P \ap a \lthen t).
\end{equation}
We now see easily that $\oracle{P}$ is a modality: it is monotone because $f^P_s$ is monotone in the parameter~$s$,
it is obviously inflationary, and it is idempotent because $\oracle{P} \ap s$ is a fixed point of~$f^P_{\oracle{P} \ap s}$.
%
%
%
Tarski~\cite{Tarski55} constructed the least fixed-point of a monotone map on a complete lattice as the infimum of its prefixed points.
In our case, $r : \Prop$ is a prefixed point of $f^P_s$ when $f^P_s \ap r \lthen r$, or equivalently, when
\begin{equation}
  \label{eq:f-prefixed}
  (s \lthen r) \lthen (\all{a \of A} (P \ap a \lthen r) \lthen r) \lthen r.
\end{equation}

We may use~\eqref{eq:f-prefixed} to check that $\oracle{P} \ap s \leq r$ holds for a given~$r$ by verifying the 
\defemph{prefixed-point conditions}  $s \lthen r$ and $\all{a \of A} ((P \ap a \lthen r) \lthen r)$. Functional programmers will recognize the technique as \emph{folding} and type theorists as \emph{induction}.
The infimum of all $r$'s satisfying~\eqref{eq:f-prefixed} is equivalent to $\oracle{P} \ap s$:
\begin{equation}
  \label{eq:forcing-modality}%
  \oracle{P} \ap s \;\liff\;
  \all{r \of \Prop} (s \lthen r) \lthen (\all{a \of A} (P \ap a \lthen r) \lthen r) \lthen r.
\end{equation}
The above formula is known~\cite{hyland1982,oosten08:_realiz} to give the least modality~$\j$ that \defemph{forces}~$P$, by which we mean that $\all{a \of A} \j(P \ap a)$ holds.
We say that a predicate $P : A \to \Prop$ is \defemph{$\j$-dense} when it is forced by~$\j$, and similarly that a truth value $p : \Prop$ is $\j$-dense when $\j \ap p$ holds.

\begin{proposition}
  \label{prop:j-forces-P-iff}%
  A modality $\j$ forces a predicate~$P : A \to \Prop$ if, and only if, $\oracle{P} \leq \j$.
\end{proposition}

\begin{proof}
  If $\oracle{P} \leq \j$ then $\j$ forces~$P$ because~$\oracle{P}$ does, which is easily checked.
  Conversely, assume $\j$ forces~$P$ and let $s : \Prop$.
  To prove $\oracle{P} \ap s \leq \j \ap s$, we check the prefixed-point conditions:
  $s \lthen \j \ap s$ holds because~$\j$ is inflationary, and if $a : A$ satisfies $P \ap a \lthen \j \ap s$ then from $\j (P \ap a)$ we deduce $\j (\j \ap s)$, whence $\j \ap s$ by idempotency of~$\j$.
\end{proof}

The following corollary tells us how to compare oracle modalities. In \cref{sec:modalities-realizability} it will be used to show that in realizability the order on oracle modalities coincides with (a generalization of) Turing reducibility.

\begin{corollary}
  \label{cor:oracle-leq}
  For predicates $P : A \to \Prop$ and $Q : B \to \Prop$, $\oracle{P} \leq \oracle{Q}$ if, and only if,
  $\all{a \of A} \oracle{Q}(P \ap a)$.
\end{corollary}

In general, the modality $\oracle{P}$ is what it is, but in some cases it can be computed.

\begin{example}
  \label{exa:modality-lem}%
  The excluded middle oracle $\lem$ from \cref{exa:oracle-lem} induces the double negation modality, $\oracle{\lem} s \liff \neg\neg s$.
  Indeed, for any $s : \Prop$ one shows $\oracle{\lem} \ap s \lthen \neg\neg s$ by giving intuitionistic proofs of the prefixed-point conditions, while the converse $\neg\neg s \lthen \oracle{\lem} \ap s$ follows from $\neg\neg s \lthen (s \lor \neg s \lthen s) \lthen s$, which also has an intuitionistic proof.
\end{example}

\begin{example}
  \label{exa:modality-counterexample}%
  If a predicate $P : A \to \Prop$ has a counterexample~$a : A$, as in \cref{exa:oracle-counterexample}, then~$\oracle{P}$ is the trivial modality $\oracle{P} \ap s \liff \top$ because $P \ap a \to s$ holds vacuously.
\end{example}

\begin{example}
  \label{exa:modality-realized}%
  If a predicate $P : A \to \Prop$ is realized by some $f : \all{a \of A} P \ap a$, as in \cref{exa:oracle-realized}, then $\oracle{P}$ is the other trivial modality, namely, $\oracle{P} \ap s \liff s$.
  Indeed, $s \lthen \oracle{P} \ap s$ holds because~$\oracle{P}$ is inflationary, while $\oracle{P} \ap s \lthen s$
  follows from the prefixed-point conditions.
\end{example}

A predicate $P : A \to \Prop$ has no counterexamples precisely when it is \defemph{$\neg\neg$-dense}, i.e., when $\all{a \of A} \neg\neg P \ap a$. \Cref{exa:modality-counterexample} teaches us that only $\neg\neg$-dense predicates beget interesting modalities, but it would be a mistake to interpret \cref{exa:modality-realized} as saying that $\neg \all{a \of A} P \ap a$ is necessary for having an interesting modality.

\begin{example}
  \label{exa:generic-point}%
  For a more elaborate example, we draw inspiration from locale theory,
  and refer to~\cite[\S2.11]{johnstone82:_stone} as a standard reference.
  Consider a meet-semilattice $(A, {\sqsubseteq}, 1, {\sqcap})$, whose elements are construed as the basic opens of a space.
  Let $\down{a} \defeq \dsum{b \of A} b \sqsubseteq a$ be the lower set of~$a : A$,
  and $\_ \triangleleft \_ : \dprod{a \of A} (\down{a} \to \Prop) \to \Prop$ a coverage on~$A$, i.e,
  if $a \triangleleft U$ and $a \sqcap b = b$ then $b \triangleleft (\lam{c} \some{d \of \down{a}} U \ap d \land c = d \sqcap b)$.
  The intuitive meaning of $a \triangleleft U$ is ``the basic open $a$ is covered by the collection of basic opens $U$''.
  A \emph{generic point} of the locale generated by~$(A, {\triangleleft})$ may be modeled as follows. Define the predicate $\pt : A \to \Prop$ inductively by stipulating $\pt \ap 1$ and $\pt \ap (a \sqcap b)$ for all $a, b : A$ such that $\pt \ap a$ and $\pt \ap b$.
  We write $\pt \in a$ instead of $\pt \ap a$ to suggest the reading ``the generic point is an element of the basic open $a$''.
  The coverage induces an oracle
  $
    (\dsum{a \of A} \dsum{U : \down{a} \to \Prop} a \triangleleft U) \to \Prop,
  $
  defined by
  \begin{equation*}
      (a, U, \_) \mapsto (a \in \pt \lthen \some{b \of \down{a}} U \ap b \land b \in \pt).
  \end{equation*}
  In words, if $\pt \in a$ and $a$ is covered by~$U$, then we ask the oracle to demonstrate that $\pt \in b$
  for some $b : \down{a}$ such that $U \ap b$. The resulting modality may be non-trivial even if the locale generated by $(A, {\triangleleft})$ has no points.
\end{example}

\subsection{A caveat about open modalities}
\label{sec:open-modalities}

In informal mathematics, reliance on a reasoning principle~$\all{a \of A} P \ap a$ in a proof of a statement~$\all{b \of B} Q \ap b$ is usually indicated by a phrase such as ``we assume~$\all{a \of A} P \ap a$''. Taken literally, this would mean that the proven statement has the form
\begin{equation}
  \label{eq:Pa-lthen-Qb}
  (\all{a \of A} P \ap a) \lthen (\all{b \of B} Q \ap b),
\end{equation}
which can be expressed as the modal statement $\open{(\all{a \of A} P \ap a)} (\all{b \of B} Q \ap b)$ using the \defemph{open modality} $\open{p} : \Prop \to \Prop$, defined for $p : \Prop$ as
$\open{p} \ap s \defeq (p \lthen s)$.
The accompanying proof typically establishes~\eqref{eq:Pa-lthen-Qb} not by using the assumption $\all{a \of A} P \ap a$ arbitrarily, but by providing a \emph{functional instance reduction}~\cite{ahman25:_comod,Bauer:InstanceReducibility}---a map $h : B \to A$ such that $\all{b \of B} (P \ap (h \ap b) \lthen Q \ap b)$, which can be used to prove the stronger statement $\all{b \of B} \oracle{P} (Q \ap b)$.

Such apparent uses of the open modality in informal mathematics thus typically arise not from a principled logical distinction but from imprecision in exposition. The oracle modality is therefore the more faithful interpretation of actual mathematical practice, especially since the accompanying proofs already conform to it.

A second reason for preferring the oracle modality is that in realizability models $\oracle{P} \ap s$ has meaningful computational content even when $\neg\all{a \of A} P \ap a$, whereas~\eqref{eq:Pa-lthen-Qb} becomes vacuous when the antecedent is false. We shall elaborate on this point in \cref{sec:modalities-realizability}.

\subsection{Instance reducibility}
\label{sec:oracle-modal-inst}

Oracle modalities are closely related to instance reducibility, a notion defined in~\cite{Bauer:InstanceReducibility} as the constructive analogue and generalization of Weihrauch reducibility~\cite{brattka11:_weihr}.
Recall that $P : A \to \Prop$ is \defemph{instance reducible} to $Q : B \to \Prop$, written $P \ileq Q$, when $\all{a \of A} \some{b \of B} (Q \ap b \lthen P \ap a)$.
Kihara~\cite{kihara23:_lawver_tiern} notes that instance reducibility is like a single-query oracle computation.
In our setting, his observation is expressed with the aid of the map $\mathcal{i}_P : \Prop \to \Prop$, defined by
\begin{equation*}
  \mathcal{i}_P \ap s \defeq \some{a \of A} (P \ap a \lthen s),
\end{equation*}
which expresses the fact that $s$ follows from a single query to an oracle specified by~$P$.
An instance reduction $P \ileq Q$ may be stated as $\all{a \of A} \mathcal{i}_Q \ap (P \ap a)$, which we read ``$\mathcal{i}_Q$ forces~$P$''.
The map~$\mathcal{i}_P$ itself is not a modality, but it generates the corresponding oracle modality.

\begin{proposition}
  \label{prop:oracle-least-above-instance}
  For any $P : A \to \Prop$,  $\oracle{P}$ is the least modality above~$\mathcal{i}_P$.
\end{proposition}

\begin{proof}
  It is obvious that $\mathcal{i}_P \leq \oracle{P}$.
  Suppose $\j$ is a modality and $\mathcal{i}_P \leq \j$.
  For any $s : \Prop$, we check $\oracle{P} \ap s \leq \j \ap s$ by verifying the prefixed-point conditions:
  $s \lthen \j \ap s$ because~$\j$ is inflationary, and if $a : A$ satisfies $P \ap a \lthen \j \ap s$ then $\mathcal{i}_P \ap (\j \ap s)$ holds, hence $\j(\j \ap s)$ and $\j \ap s$ do as well.
\end{proof}

A category theorist would explain \cref{prop:oracle-least-above-instance} by noting that $\mathcal{i}_P$ is
the $\Prop$-valued polynomial functor for the propositional container~$P$, and that $\oracle{P}$ is the free monad generated on it~\cite{barr:coequalizers}.
In fact, the free monad is explicitly given as the least fixed-point of $f^P_s$ from \eqref{eq:oracle-map}.

\subsection{Every modality is an oracle modality}
\label{sec:every-modality-oracle}

We next study the relationship between predicates and modalities from a category-theoretic perspective.
The two categories of interest are the propositional containers~$\PCont$ and the frame of modalities~$\Mod$.
First, the mapping from the former to the latter is functorial.

\begin{proposition}
  The mapping $\oracle{}$ is a functor from $\PCont$ to $\Mod$.
\end{proposition}

\begin{proof}
  Given a morphism $\cont{f}{g} : \cont{A}{P} \to \cont{B}{Q}$ of propositional
  containers, we apply \cref{cor:oracle-leq} to reduce $\oracle{P} \leq \oracle{Q}$ to
  $\all{a \of A} \oracle{Q} \ap (P \ap a)$, which is proved by $\lam{a} \ask \ap (f \ap a) \ap (\lam{u} \prf \ap (g \ap a \ap u))$.
\end{proof}

When looking for a modality that is not an oracle modality, it occurred to us that
there isn't one, because there is a functor in the opposite direction.
Given a modality $\j$, let $\DProp_\j \defeq \dsum{s \of \Prop} \j \ap s$ be the type of \defemph{$\j$-dense propositions}.

\begin{proposition}
  Given a modality~$\j$, define $\predicate{\j}$ to be the propositional container $\cont{\DProp_\j}{I_\j}$, where $I_\j : \DProp_\j \to \Prop$ is the first projection $I_\j (s, \_) \defeq s$.
  The mapping $\predicate{}$ is a functor from $\Mod$ to $\PCont$.
\end{proposition}

\begin{proof}
  Given modalities $\j$ and $\k$, such that $\j \leq \k$, let $h
  : \DProp_\j \to \DProp_\k$ act as identity on the first component and use $\j \leq \k$ on the second component. It is immediate that $\all{(s,u) \of \DProp_\j}
  I_\k \ap (h \ap (s, u)) \lthen I_\j \ap (s,u)$, as it is just $\all{(s,u) \of \DProp_\j} s
  \lthen s$. Functoriality holds because $h$ acts as identity on the
  first component and~$\leq$ is transitive.
\end{proof}

The functors $\oracle{}$ and $\predicate{}$ are adjoint, provided we squash the morphisms in~$\PCont$.
Let $\PCont_\leq$ be the preorder reflection of~$\PCont$, i.e., $(\cont A P) \leq (\cont B Q)$ when there \emph{exists} a morphism $\cont A P \to \cont B Q$.

\begin{theorem}
  When $\oracle{}$ and $\predicate{}$ are construed as functors between $\PCont_\leq$ and $\Mod$,
  $\oracle{}$ is left adjoint to~$\predicate{}$.
  Moreover, $\predicate{}$ is a section of~$\oracle{}$.
\end{theorem}

\begin{proof}
  Given a propositional container $\cont{A}{P}$ and a modality $\j$, we need to
  check that there is a morphism $\cont{A}{P} \to \predicate{\j}$ if, and only if, $\oracle{P} \leq \j$.

  If $\oracle{P} \leq \j$, then we define $f : A \to \DProp_\j$
  by $f \ap a \defeq (P \ap a, u)$, where $u : \j \ap (P a)$ is provided by
  \cref{prop:j-forces-P-iff} because $\j$ forces~$P$. It is immediate
  that $\all{a \of A} I_\j \ap (f \ap a) \lthen P \ap a$, which is just $\all{a \of A} P \ap a
  \lthen P \ap a$.

  If there exists a morphism $\cont{f}{g} : \cont{A}{P} \to \predicate{\j}$ of
  propositional containers, then by \cref{prop:j-forces-P-iff} it suffices to prove $\all{a \of A} \j \ap (P \ap a)$.
  So consider any $a \of A$, and let $(s, u) \defeq f \ap a$ so that $s : \Prop$ and $u : \j \ap s$.
  Now $g \ap a$ establishes $s \lthen P \ap a$, hence $\j \ap (P \ap a)$ by monotonicity of~$\j$.

  It remains to show that~$\j = \oracle{\predicate{\j}}$.
  The inequality $\oracle{\predicate{\j}} \leq \j$ holds by
  \cref{prop:j-forces-P-iff} as~$\j$ forces~$\predicate{\j}$.
  To show $\j \leq \oracle{\predicate{\j}}$, consider any $s : \Prop$ and
  $u : \j \ap s$. Then $\oracle{\predicate{\j}} \ap s$ holds because $(s, u)$ witnesses $\some{(t,\_)
  \of \DProp_\j} (t \lthen \oracle{\predicate{\j}} \ap s)$.
\end{proof}

\begin{corollary}
  \label{cor:every-modality-oracle}
  Every modality is an oracle modality.
\end{corollary}

\begin{proof}
  A modality~$\j$ coincides with the oracle modality $\oracle{\predicate{\j}}$.
\end{proof}

Thus, a modality~$\j$ is the least oracle modality forcing $\j$-dense propositions.
Note that there are many predicates whose induced oracle modality
equals~$\j$, some of which might be more convenient to work with.
For instance, in \cref{exa:modality-lem} we saw that double negation is the Excluded Middle oracle modality; it seems easier to work with an oracle that decides propositions than one that removes double negations from $\neg\neg$-dense propositions.

The adjunction gives us a recipe for computing suprema of nuclei, whose direct computation is complicated~\cite{escardo03:_joins_frame_nuclei}, in terms of coproducts of (propositional) containers, which is
straightforward~\cite{AbbottAG:Containers}.

\begin{proposition}
  \label{prop:oracle-sup}
  Let $P : \dprod{i \of I} (A_i \to \Prop)$ be an $I$-indexed family of
  predicates. The supremum of $\oracle{P \ap i}$'s is the oracle modality for the
  predicate $\Sigma \ap P : (\dsum{i \of I} A_i) \to \Prop$ defined by $(\Sigma \ap P)
  (i, a) \defeq P \ap i \ap a$.
\end{proposition}

\begin{proof}
  Immediate, because $\oracle{}$ is a left adjoint and the coproduct $\cont{(\dsum{i \of I}
  A_i)}{(\Sigma \ap P)}$ of containers reflects to the supremum in~$\PCont_{\leq}$.
\end{proof}


\section{Sheaves for oracle modalities}
\label{sec:sheaves}

We now explore sheaves for oracle modalities. Our first goal is a type-theoretic characterization of sheaves for a general modality that can be usefully specialized to an oracle modality.

In a topos, an object $X$ is a $\j$-sheaf when for every $\j$-dense mono $m : B \to A$ the map ${-} \circ m : (A \to X) \to (B \to X)$ is an isomorphism~\cite[\S5.2]{lane92:_sheav_geomet_logic}.
As every such mono is isomorphic to the first projection $(\dsum{a \of A} P \ap a) \to A$, for some $\j$-dense predicate $P : A \to \Prop$, the sheaf condition can be stated in terms of precomposing with such projections.
After currying the projection $(\dsum{a \of A} P \ap a) \to A$ to $\dprod{a \of A} (P \ap a \to X)$, we see that the map required to be an equivalence is $(\lam{f \ap a \ap u} f \ap a) : (A \to X) \to \dprod{a \of A} (P \ap a \to X)$.

Since we work internally in type theory, we can make further simplifications.

\begin{proposition}
  \label{prop:sheaf-tfae}%
  Given a modality~$\j$ and a type~$X$, the following are equivalent:
  \begin{enumerate}
  \item
    For every $\j$-dense $P : A \to \Prop$, the map
    $(\lam{f \ap a \ap u} f \ap a) : (A \to X) \to \dprod{a \of A} P \ap a \to X$ is an equivalence.
  \item For every $\j$-dense $s : \Prop$, the map
    $(\lam{x \ap u} x) : X \to (s \to X)$ is an equivalence.
  \item The following conditions hold:
    \begin{enumerate}
    \item \label{itm:sheaf-e} there is a map $e : \dprod{s \of \Prop} \j \ap s \to (s \to X) \to X$,
    \item \label{itm:sheaf-e-eq} $e \ap s \ap t \ap h = h \ap u$ for all $s : \Prop$ and $t : \j \ap s$ and $h : s \to X$ and $u : s$, and
    \item \label{itm:sheaf-xy-eq} if $s \lthen x = y$ then $x = y$, for all $\j$-dense $s : \Prop$ and $x, y : X$.
    \end{enumerate}
  \end{enumerate}
  When these statements hold, we say that $X$ is a \defemph{$\j$-sheaf}.
\end{proposition}

\begin{proof}
  The first statement instantiated at $A = \one$ yields the second statement.

  The third statement is just an elaboration of the second statement. Indeed, given a $\j$-dense $s : \Prop$, the map $(\lam{x \ap u} x) : X \to (s \to X)$ is an equivalence when its fibers are contractible. The map~$e$ in the third statement inhabits the fibers, while the second condition states that every fiber has at most one element, because $\lam{(u \of s)} x = \lam{(u \of s)} y$ is equivalent to $s \lthen x = y$.

  We assume the third statement and prove the first one for a given $\j$-dense $P : A \to \Prop$.
  \Cref{itm:sheaf-e-eq} implies that the fibers of $\lam{f \ap a \ap u} f \ap a$ are inhabited by the values of $\lam{h \ap a} e \ap (P \ap a) \ap \_ \ap (h \ap a)$.
  %
  To see that each fiber has at most one element, consider $f, g : A \to X$ such that $\lam{a \ap u} f \ap a = \lam{a \ap u} g \ap a$. Then for any $a : A$ we have
  $\lam{u} f \ap a = \lam{u} g \ap a$, hence $P \ap a \lthen f \ap a = g \ap a$ and $f \ap a = g \ap a$ by \cref{itm:sheaf-xy-eq}. We conclude $f = g$ by function extensionality.
\end{proof}
 
The following theorem specializes the third statement in the lemma to oracle modalities in a way that avoids explicit reference to the modality, which simplifies manipulation of oracle sheaves.

\begin{theorem}
  \label{thm:oracle-sheaf-when}%
  A type $X$ is an $\oracle{P}$-sheaf for $P : A \to \Prop$ if, and only if:
  \begin{enumerate}
  \item \label{itm:oracle-d} there is a map $d : \dprod{a \of A} (P \ap a \to X) \to X$,
  \item \label{itm:oracle-d-eq} $d \ap a \ap h = h \ap u$ for all $a : A$ and $h : P \ap a \to X$ and $u : P \ap a$, and
  \item \label{itm:oracle-xy-eq} if $P \ap a \lthen x = y$ then $x = y$, for all $a : A$ and $x, y : X$.
  \end{enumerate}
\end{theorem}

\begin{proof}
  If $X$ is an $\oracle{P}$-sheaf, we use~$e$ from \cref{prop:sheaf-tfae} to define
  $
    d \ap a \ap h \defeq e \ap (P \ap a) \ap (\ask \ap a \ap \prf) \ap h
  $.
  It is not hard to check that~$d$ satisfies \cref{itm:oracle-d-eq},
  while \Cref{itm:oracle-xy-eq} follows from \cref{itm:sheaf-xy-eq} in \cref{prop:sheaf-tfae}, as $P \ap a$ is $\oracle{P}$-dense.

  To prove the converse, we assume $X$ satisfies \cref{itm:oracle-d,itm:oracle-d-eq,itm:oracle-xy-eq} and show
  that it also satisfies \cref{itm:sheaf-e,itm:sheaf-e-eq,itm:sheaf-xy-eq} from \cref{prop:sheaf-tfae}.
  Given $s : \Prop$, and $t : \oracle{P} \ap s$, one is tempted to define $e \ap s \ap t$ by induction on~$t$,
  but this is forbidden as~$s \to X$ is not known to be a proposition. Instead, we prove by induction on~$t$ that
  $
    \all{h \of s \to X} \someone{x \of X} \all{u \of s} x = h \ap u
  $,
  and use unique choice to extract $e \ap s \ap t : (s \to X) \to X$ satisfying $e \ap s \ap t \ap h = h \ap u$ for all $h : s \to X$ and $u : s$, thereby also securing \cref{itm:sheaf-e-eq}. We omit the rather technical details, which we checked
  with Rocq.
  Finally, \Cref{itm:sheaf-xy-eq} follows from \cref{itm:oracle-xy-eq} and the next statement:
  given $q : \Prop$, if $\all{a \of A} (P \ap a \to q) \to q$ then $\all{s \of \Prop} \oracle{P} \ap s \to (s \to q) \to q$.
  This time we may use induction.
\end{proof}

Henceforth we refer to $\oracle{P}$-sheaves as \defemph{$P$-oracle sheaves} or just \defemph{$P$-sheaves}.
We refer to the map~$d$ from~\Cref{thm:oracle-sheaf-when} as the \defemph{structure map} of the sheaf~$X$,
and to \cref{itm:oracle-d-eq,itm:oracle-xy-eq} as the first and second \defemph{sheaf equality}, respectively.
The terminology is warranted because the theorem exhibits a sheaf as an algebra equipped with an operation~$d$ taking inputs $a : A$ and $h : P \ap a \to X$, and extracting the element that~$h$ would compute when given an answer $u : P \ap a$, but \emph{without} consulting the oracle. Such an~$X$ must be quite special when~$P$ is non-trivial.
The second sheaf equality, \Cref{itm:oracle-xy-eq}, similarly allows us to extract the equality $x = y$ from the hypothetical equality $P \ap a \lthen x = y$, again without consulting the oracle.
We also observe that being a sheaf is a property, not structure, as the map~$d$ is unique when it exists.

We reprove some known facts about sheaves to practice working with oracle sheaves.

\begin{corollary}
  \label{cor:propj-sheaf}%
  The type of \defemph{$\oracle{P}$-stable propositions} $\Prop_P \defeq \dsum{s \of \Prop} (\oracle{P} \ap s \lthen s)$ is a $P$-sheaf and it is equivalent to $\dsum{s \of \Prop} (\all{a \of A} (P \ap a \lthen s) \lthen s)$.
\end{corollary}

\begin{proof}
  The structure map $d : \all{a \of A} (P \ap a \to \Prop_P) \to \Prop_P$ is
  $d \ap a \ap h \defeq ((\all{u \of P \ap a} h \ap u), \_)$, where $\_$ should be filled with a proof of
  $\oracle{P} (\all{u \of P \ap a} h \ap u) \lthen \all{u \of P \ap a} h \ap u$. The proof, as well as the verification of sheaf equalities, are exercise in induction on proofs of modal statements.

  For the equivalence to be established, it suffices to show, for any $s : \Prop$, that $\oracle{P} \ap s \lthen s$ is equivalent to $\all{a \of A} (P \ap a \lthen s) \lthen s$.
  If $\all{a \of A} (P \ap a \lthen s) \lthen s$ and $t : \oracle{P} \ap s$, then $s$ follows by induction on~$t$: if $t = \prf \ap u$ then $u : s$, and  if $t = \ask \ap a \ap \kappa$ then the induction hypothesis states $P \ap a \lthen s$, therefore $s$ by assumption. The converse holds because $P \ap a$ is $\oracle{P}$-dense.
\end{proof}

Recall that $X$ is \defemph{$\j$-separated} when equality on~$X$ is $\j$-stable, $\all{x, y \of X} \j (x = y) \lthen x = y$.

\begin{corollary}
  \label{cor:P-sheaves-P-separated}%
  $P$-sheaves are $P$-separated.
\end{corollary}

\begin{proof}
  Let $X$ be a $P$-sheaf for a predicate $P : A \to \Prop$ and $x, y : X$ such that $t : \oracle{P} \ap (x = y)$.
  We prove $x = y$ by induction on~$t$. If $t = \prf \ap u$ then $u : x = y$.
  If $t = \ask \ap a \ap \kappa$, then the induction hypothesis states $P \ap a \to x = y$ and so $x = y$ by
  the second sheaf equality.
\end{proof}

The next corollary says that all maps between sheaves are homomorphisms.

\begin{corollary}
  \label{cor:sheaf-maps-homomorphisms}%
  Let $X$ and $Y$ be $P$-sheaves for $P : A \to \Prop$, with structure maps~$d_X$ and $d_Y$, respectively.
  Then $f (d_X \ap a \ap h) = d_Y \ap a \ap (f \circ h)$ for all $f : X \to Y$, $a : A$ and $h : P \ap a \to X$.
\end{corollary}

\begin{proof}
  We employ the second sheaf equality for~$Y$. If $u : P \ap a$ then the first sheaf equality implies $f \ap (d_X \ap a \ap h) = f \ap (h \ap u) = (f \circ h) \ap u = d_Y \ap a \ap (f \circ h)$.
\end{proof}

\begin{corollary}
  For any predicate $P : A \to \Prop$, $P$-sheaves are closed under dependent products.
\end{corollary}

\begin{proof}
  Algebras are closed under products. Specifically, if $X : I \to \Type$ is a family of sheaves with structure maps $d : \dprod{i \of I} \dprod{a \of A} (P \ap a \to X_i) \to X_i$, then $Y \defeq \dprod{i \of I} X_i$ is a $P$-sheaf whose structure map $e : \dprod{a \of A} (P \ap a \to Y) \to Y$ is defined index-wise by
  $e \ap a \ap h \ap i \defeq d_i \ap a \ap (\lam{u} h \ap u \ap i)$. The sheaf equations are derived index-wise, too.
\end{proof}

\begin{example}
  If $P : A \to \Prop$ has a counterexample $a : A$, so that $\neg P \ap a$, then the $P$-sheaves are precisely the contractible types. For if $X$ is a $P$-sheaf with structure map~$d$, then it is inhabited by $d \ap a \ap \_$, while
  the second sheaf equation instantiated at~$a$ yields $x = y$, for any $x, y : X$.
\end{example}

\begin{example}
  If $P : A \to \Prop$ has a section $f : \all{a \of A} P \ap a$, then every type~$X$ is a $P$-sheaf because
  its structure map can be constructed as $\lam{a \ap h} h \ap (f \ap a)$. The sheaf equations are easily checked.
  In other words, the section $f : \all{a \of A} P \ap a$ implements a $P$-oracle, so any oracle questions can be answered by just consulting $f$.
\end{example}

\begin{example}
  For any predicate $P : A \to \Prop$ and type $X$, the type $(\all{b \of A} P \ap b) \to X$ is a $P$-sheaf.
  To satisfy \cref{thm:oracle-sheaf-when}, define
  $d : \dprod{a \of A} (P \ap a \to (\all{b \of A} P \ap b) \to X) \to (\all{b \of A} P \ap b) \to X$ by
  $
    d \ap a \ap h \ap f \defeq h \ap (f \ap a) \ap f
  $.
  The two sheaf equations are easily checked.
  %
  %
  In fact, the general fact that a $\k$-sheaf is a $\j$-sheaf if $\j \leq \k$ is at work here,
  with the two modalities being $\oracle{P}$ and $\open{(\all{a \of A} P \ap a)}$ from~\cref{sec:open-modalities}.
\end{example}

\subsection{Sheafification}
\label{sec:sheafification}

In a topos, sheaves form a reflective subcategory~\cite[\S5.3]{lane92:_sheav_geomet_logic}. This is the case in type theory, too, provided the setup is rich enough to allow sheafification, which is described by the following universal property.

\begin{definition}
  \label{def:sheafification}%
  Given a modality~$\j$, the \defemph{$\j$-sheafification} of a type~$X$ is a $\j$-sheaf $\sheafify[\j]{X}$
  with a map $\eta_X : X \to \sheafify[\j]{X}$, such that for every $\j$-sheaf $Y$ and $f : X \to Y$ there
  is a unique map $f^\dagger : \sheafify[\j]{X} \to Y$ for which $f = f^\dagger \circ \eta_X$.
\end{definition}

The sheafification of~$X$ is unique up to isomorphism if it exists, but how do we construct it?
One way is to translate ``sheafification in logical form'' by Barbara Veit~\cite{veit81} from higher-order logic to type theory.
After some calculations, which rely on having an impredicative~$\Prop$, propositional extensionality, and proof-irrelevance, we find out that
\begin{equation}
  \label{eq:sheafify-logical-form}%
  \sheafify[\j]{X} \defeq \dsum{p : X \to \Prop_\j} \j (\some{x \of X} \all{y \of X} p \ap y \liff \j (x = y)).
\end{equation}
The formula says that the elements of $\sheafify[\j]{X}$ are the ``$\j$-modally $\j$-modal singletons of~$X$'', for erasing~$\j$ from it yields ordinary singleton subsets of~$X$.
Alternatively, if we treat $P$-sheaves as algebras in the sense of \cref{thm:oracle-sheaf-when}, then \cref{def:sheafification} says that $\sheafify[\j]{X}$ is the initial algebra generated by~$X$, which is constructed as follows.

\begin{definition}
  \label{def:sheafification-qit}%
  Given a predicate $P : A \to \Prop$, the \defemph{$P$-sheafification} of $X$ is the quotient-inductive type $\sheafify[P]{X}$, generated by the following constructors and equations:
  \begin{itemize}
  \item $\ret : X \to \sheafify[P]{X}$,
  \item $\ask : \dprod{a \of A} (P \ap a \to \sheafify[P]{X}) \to \sheafify[P]{X}$,
  \item $\ask \ap a \ap h = h \ap u$, for all $a : A$, $h : P \ap a \to \sheafify[P]{X}$, and $u : P a$, and
  \item if $P \ap a \lthen v = w$ then $v = w$, for all $a : A$ and $v, w : \sheafify[P]{X}$.
  \end{itemize}
  The induction principle for $\sheafify[P]{X}$ states: given a sheaf $Y$ with structure map $d_Y$ and $f : X \to Y$, there is a unique map $f^\dagger : \sheafify[P]{X} \to Y$ such that $f^\dagger \ap (\ret \ap x) = f \ap x$.
\end{definition}

By \emph{quotient-inductive type} we mean a type that is generated inductively by its constructors and simultaneously quotiented by its equations. Such types are instances of higher inductive types in homotopy type theory~\cite{hottbook}, and were specifically studied by Thorsten Altenkirch and Ambrus Kaposi~\cite{altenkirch16:_type}. To be precise, we require a variant of \emph{conditional quotient-inductive types} formulated by Peter Dybjer and Hugo Moeneclaey~\cite{dybjer18:_finit_hit}, and briefly considered by Marcelo Fiore, Andrew Pitts, and S.~C.~Steenkamp~\cite{FiorePS22}, because the second sheaf equation is conditioned by a $P \ap a$-indexed family of equations.
We could also construct $P$-sheafification by \emph{nullification} of Egbert Rijke, Mike Shulman, and Bas Spitters~\cite[\S2.3]{rijke-shulman-spitters-2020}, which is the least modality inverting the maps $P \ap a \to \one$, for all $a \of A$. They too construct the nullification as a higher-inductive type.
We shall not dwell further on foundational assumptions that guarantee the existence of sheafification---suffice it to say that~\eqref{eq:sheafify-logical-form} guarantees its existence under our working assumptions.

Observe that in view of \cref{cor:sheaf-maps-homomorphisms}, we need not state that $f^\dagger$ provided by the induction principle for $\sheafify[P]{X}$ is a homomorphism with respect to the structure maps, but if we did, we would write down the equation $f^\dagger \ap (\ask \ap a \ap \kappa) = d_Y \ap a \ap (f^\dagger \circ \kappa)$.

We are tempted to think of the elements of $\sheafify[P]{X}$ as effectful computations on~$X$, where $\ret \ap x$ is the pure computation returning~$x : X$, and $\ask \ap a \ap \kappa$ is a computation that queries~$a : A$ and proceeds as $\kappa \ap u$ upon receiving evidence $u : P \ap a$ as an answer. The first equation says that a query can be skipped if the answer is already available, and the second one says that we may assume answers to queries when proving equality of computations.

\begin{example}
  For any $P : A \to \Prop$, a map $f : X \to \sheafify[P]{Y}$ can be thought as computing its values with the aid of a $P$-oracle. Andrew Swan~\cite{swan24:_oracl} takes such maps as the basic notion. For example, the excluded middle oracle $\lem$ from \cref{exa:oracle-lem} can be used to define the halting oracle $H : \NN \times \NN \to \sheafify[\lem]{\NN}$ by
  \begin{equation*}
    H \ap (m, n) \defeq \ask \ap (\some{k \of \NN} T(m, n, k)) \ap
     (\lam{b} \;\mathsf{if}\; b \;\mathsf{then}\; \ret \ap 1 \;\mathsf{else}\; \ret \ap 0).
  \end{equation*}
  To decide whether the $n$-th Turing machine stops on input~$m$, $H$ uses Kleene's~$T$ predicate~\cite{kleene43:_recur} to ask the oracle whether there is an execution trace~$k$ witnessing termination, and answers accordingly.
\end{example}

\subsection{Equifoliate trees}
\label{sec:oracle-modal-pred}

Treating the elements of $\sheafify[P]{X}$ as $P$-oracle computations works only so well, because the quotient collapses intensional distinctions that one normally expects under the classic treatment of computations modeled by trees~\cite{PlotkinP:NotionsOfComputations}.
The sheaf equations mask information about which queries are asked, and also, computations are entirely streamlined, as the answer type $P \ap a$ has at most one element.
For example, if the oracle answers an existential statement $\some{b \of B} Q \ap b$ by providing a witness~$b$ and a proof of~$Q \ap b$, the computation may use~$b$ only in a limited fashion because neither further queries nor the returned element may vary with~$b$.

A computationally more faithful account of oracle computations would specify an oracle by a \emph{non-propositional container} $P : A \to \Type$ and then model $P$-oracle computations returning values of type~$X$ as trees $\Tree[P] X$, generated inductively by the constructors:
\begin{itemize}
\item $\leaf : X \to \Tree[P] X$,
\item $\node : \dprod{a \of A} (P \ap a \to \Tree[P] X) \to \Tree[P] X$.
\end{itemize}

It is natural to ask how $\Tree[P] X$ corresponds to the sheafification $\sheafify[\Trunc P] X$, where
$\Trunc P : A \to \Prop$ is the \defemph{propositional truncation} of~$P$, which may be defined as $\Trunc P a \defeq \some{u \of P \ap a} \top$.
Without further assumptions there is no connection: a tree $t \in \Tree[P] X$ may compute different values along different paths, while attempts to lift an element of $\sheafify[\Trunc P] X$ to a tree will generally require some form of the axiom of choice. We address the former problem first.

\begin{definition}
  \label{def:mem-shfy}%
  Given $P : A \to \Prop$ and a type $X$, define the \defemph{membership} relation $x \in t$ for $x : X$ and $t : \sheafify[P] X$ inductively as follows:
  \begin{itemize}
  \item $x \in \ret \ap y$ if $\oracle{P} \ap (x = y)$,
  \item $x \in \ask \ap a \ap \kappa$ if $\all{u \of P \ap a} x \in \kappa \ap u$.
  \end{itemize}
\end{definition}

For \cref{def:mem-shfy} to be well-formed, we need to check that $\in$ respects the sheaf equations.
For the first equation, take any $a : A$, $\kappa : P \ap a \to \sheafify[P] X$ and $u : P \ap a$,
and observe that $x \in \ask \ap a \ap \kappa \liff (\all{v \of P \ap a} x \in \kappa \ap v) \liff x \in \kappa \ap u$ because all elements of $P \ap a$ are equal.
For the second equation, observe that $x \in t$ is $\oracle{P}$-stable by construction.

The occurrence of $\oracle{P}$ in the first clause of \cref{def:mem-shfy} is annoying but it is needed, for in general $\ret \ap x = \ret \ap y$ implies only $\oracle{P} (x = y)$.
Luckily, in many concrete cases $X$ is $P$-separated, in which case the application of~$\oracle{P}$ is 
 unnecessary; one such instance occurs when~$P$ is $\neg\neg$-dense and~$X$ has $\neg\neg$-stable equality.
The second clause also appears odd, for would it not be more natural to require that~$x$ appear in \emph{some}
subtree~$\kappa \ap u$ instead of all of them?
No, for if $\neg P \ap a$ holds for some $a : A$, then by the second sheaf equation $\ask \ap a \ap \kappa$ and $\ret \ap x$ are equal, for all $x : X$.



\begin{lemma}
  \label{lem:elem-to-ret-eq}%
  For any $t : \sheafify[P] X$ and $x : X$, if $x \in t$ then $\ret \ap x = t$.
\end{lemma}

\begin{proof}
  The proof proceeds by induction on~$t$.
  If $t = \ret \ap y$ then $\oracle{P} \ap (x = y)$, therefore $\oracle{P} \ap (\ret \ap x = t)$ and
  $\ret \ap x = t$ because sheaves are separated.
  If $t = \ask \ap a \ap \kappa$ then $\all{u \of P \ap a} x \in \kappa \ap u$, hence by induction hypothesis and the first sheaf equation $\all{u \of P \ap a} \ret \ap x = \kappa \ap u = t$, from which $\ret \ap x = t$ by the second sheaf equation.
\end{proof}

\begin{lemma}
  \label{lem:sheaf-ext}%
  If $s, t : \sheafify[P] X$ are such that $\all{x \of X} (x \in s \liff x \in t)$ then $s = t$.
\end{lemma}

\begin{proof}
  The proof proceeds by induction on $s$ and $t$.
  If $s = \ret \ap x$ then $x \in t$ and we may apply \cref{lem:elem-to-ret-eq}. The case $t = \ret \ap y$ is symmetric.
  Consider $s = \ask \ap a \ap \kappa$ and $t = \ask \ap b \ap \mu$. By the second sheaf equation it suffices to show $s = t$ assuming $u : P \ap a$ and $v : P \ap b$. With~$u$ and~$v$ in hand, $\all{x \of X} x \in s \liff x \in \kappa \ap u$ and $\all{x \of X} x \in t \liff x \in \mu \ap v$, whence $\all{x \of X} x \in \kappa \ap u \liff x \in \mu \ap v$, and so by the induction hypothesis and the first sheaf equation $s = \kappa \ap u = \kappa \ap v = t$.
\end{proof}

The corresponding notion of membership for non-propositional containers is as follows.

\begin{definition}
  Given $P : A \to \Type$ and a type $X$, define the \defemph{membership} relation $x \in t$ for $x : X$ and $t : \Tree[P] X$ inductively as follows:
  \begin{itemize}
  \item $x \in \leaf \ap y$ if $\oracle{\Trunc P} \ap (x = y)$
  \item $x \in \node \ap a \ap \kappa$ if $\all{u \of P \ap a} x \in \kappa \ap u$.
  \end{itemize}
\end{definition}

We now identify those computation trees that compute just one element, modally speaking.

\begin{definition}
  Given $P : A \to \Type$ and a type~$X$, define $\equi : \Tree[P] X \to \Prop$ inductively by the clauses:
  \begin{itemize}
  \item $\equi \ap (\leaf \ap x)$, for all $x : X$,
  \item $\equi \ap (\node \ap a \ap \kappa)$, when $\all{u \of P \ap a} \equi \ap (\kappa \ap u)$ and
    $\all{x \of X} \all{u, v \of P \ap a} x \in \kappa \ap u \lthen x \in \kappa \ap v$.
  \end{itemize}
  An \defemph{equifoliate tree} is one satisfying~$\equi$.
 Let $\EquiTree[P] X \defeq \dsum{t \of \Tree[P] X} \equi \ap t$ be the type of equifoliate trees.
\end{definition}

The nomenclature is sensible because an equifoliate tree computes a single element, modally speaking:

\begin{lemma}
  Given $P : A \to \Type$, a type~$X$, an equifoliate tree $t : \Tree[P] X$, and $x, y : X$, if $x \in t$ and $y \in t$ then $\oracle{\Trunc P} \ap (x = y)$.
\end{lemma}

\begin{proof}
  We proceed by induction on~$t$. If $t$ is a leaf, the statement is immediate.
  If $t = \node \ap a \ap \kappa$ then $P \ap a \lthen \oracle{\Trunc P} \ap (x = y)$ is the induction hypothesis, from which $\oracle{\Trunc P} \ap (x = y)$ follows.
\end{proof}

For equifoliate trees to be a reasonable notion of sheaf computations, they should form a monad, which they do.

\begin{lemma}
  \label{lem:mem-bind}%
  Given $P : A \to \Type$, $f : X \to \Tree[P] Y$, $t : \EquiTree[P] X$, and $y : Y$,
  we have $y \in \bind \ap f \ap t$ if, and only if, $\all{x \of X} x \in t \lthen y \in f \ap x$.
\end{lemma}

\begin{proof}
  We proceed by induction on~$t$.

  Consider the case $t = \leaf \ap z$.
  If $y \in \bind \ap f \ap t = f \ap z$ and $x \in t$ then $\oracle{\Trunc P} \ap (x = z)$, so $\oracle{\Trunc P} \ap (y \in f \ap x)$, hence $y \in f \ap x$.
  Conversely, if $\all{x \of X} x \in t \lthen y \in f \ap x$ then $y \in f \ap z = \bind \ap f \ap t$ because $z \in \leaf \ap z$.

  Consider the case $t = \node \ap a \ap \kappa$.
  If $y \in \bind \ap f \ap t = \node \ap a (\lam{u} \bind \ap f \ap (\kappa \ap u))$ and $x \in t$ then, for any $u : P \ap a$, $x \in \kappa \ap u$ hence $y \in f \ap x$ by the induction hypothesis and the fact that $y \in \bind \ap f \ap (\kappa \ap u)$.
  Conversely, suppose $\all{x \of X} x \in t \lthen y \in f \ap x$.
  It suffices to show that, for a given $u : P \ap a$, we have $y \in \bind \ap f \ap (\kappa \ap u)$.
  By the induction hypothesis this follows once we demonstrate $\all{x \of X} x \in \kappa \ap u \lthen y \in f \ap x$. But if $x \in \kappa \ap u$ then $\all{v \of P \ap a} x \in \kappa \ap v$ because~$t$ is equifoliate, hence $x \in t$ and $y \in f \ap x$ by assumption.
\end{proof}

\begin{theorem}
  The tree monad on $\Tree[P]$ restricts to~$\EquiTree[P]$, for any $P : A \to \Type$.
\end{theorem}

\begin{proof}
  The unit of the tree monad is~$\leaf$, and leaves are equifoliate.
  To show that $\bind$ restrict to equifoliate trees as well, consider any $f : X \to \EquiTree[P] Y$ and $t : \EquiTree[P] X$. We show that $\bind \ap f \ap t$ is equifoliate by induction on~$t$.

  If $t = \leaf \ap x$ then $\bind \ap f \ap t = f \ap x$ is equifoliate by assumption.

  Consider the case $t = \node \ap a \ap \kappa$, so $\bind \ap f \ap t = \node \ap a \ap (\lam{u} \bind \ap f \ap (\kappa \ap u))$.
  The subtrees of $\bind \ap f \ap t$ are equifoliate by the induction hypothesis.
  Given $y : Y$ and $u, v : P \ap a$ such that $y \in \bind \ap f \ap (\kappa \ap u)$, we need to show $y \in \bind \ap f \ap (\kappa \ap v)$. By \cref{lem:mem-bind}, $\all{x \of X} x \in \kappa \ap u \lthen y \in f \ap x$.
  Because $t$ is equifoliate, this implies $\all{x \of X} x \in \kappa \ap v \lthen y \in f \ap x$, and then
  \cref{lem:mem-bind} gets us the desired $y \in \bind \ap f (\kappa \ap v)$.
\end{proof}

\begin{theorem}
  Given $P : A \to \Type$, there is a map $\updelta_P : \EquiTree[P] X \to \sheafify[\Trunc P] X$ such that, for all $x : X$ and $t : \EquiTree[P] X$, $x \in t \liff x \in \updelta_P \ap t$.
\end{theorem}

\begin{proof}
  We define $\updelta_P$ by induction on its argument and simultaneously show that it has the required property.
  For any $y \of X$, define $\updelta_P \ap (\leaf \ap y) \defeq \ret \ap y$ and note that $x \in \leaf \ap y \liff \oracle{\Trunc P} \ap (x = y) \liff x \in \ret \ap y$, for any $x : X$.

  Consider an equifoliate tree $\node \ap a \ap \kappa$, where $a : A$ and $\kappa : P \ap a \to \EquiTree[P] X$, and suppose $\updelta_P (\kappa \ap u)$ is already defined for all $u \of P \ap a$, and also $\all{x \of X} (x \in \kappa \ap u \liff x \in \updelta_P \ap (\kappa \ap u))$.
  We claim that for all $u, v : P \ap a$, we have $\updelta_P \ap (\kappa \ap u) = \updelta_P \ap (\kappa \ap v)$:
  indeed, by the assumption and \cref{lem:sheaf-ext} it suffices to show $\all{x \of X} (x \in \kappa \ap u \liff x \in \kappa \ap v)$, which holds because~$\node \ap a \ap \kappa$ is equifoliate.
  The claim implies that $\updelta_P \circ \kappa$ factors through the quotient map $P \ap a \to \Trunc {P \ap a}$
  as $\mu : \Trunc {P \ap a} \to \sheafify[\Trunc P] X$. We now define $\updelta_P \ap (\node \ap a \ap \kappa) \defeq
  \ask \ap a \ap \mu$.
\end{proof}

The map $\updelta_P$ is surjective, modally.

\begin{theorem}
  Given $P : A \to \Type$, the map $\updelta_P : \EquiTree[P] X \to \sheafify[\Trunc P] X$ is modally surjective:
  $\all{t \of \sheafify[\Trunc P] X} \oracle{\Trunc P} \ap (\some{e \of \EquiTree[P] X} \updelta_P \ap e = t)$.
\end{theorem}

\begin{proof}
  We proceed by induction on $t \of \sheafify[\Trunc P] X$.
  If $t = \ret \ap x$ then $\updelta_P \ap (\leaf \ap x) = \ret \ap x$, so we just apply $\eta$ to~$\leaf \ap x$.
  If $t = \ask \ap a \ap \kappa$, since $\oracle{\Trunc P}$ forces $\Trunc{P \ap a}$, we obtain modally a witness $w \of \Trunc{P \ap a}$ and, by the induction hypothesis at~$w$, some~$e$ with $\updelta_P \ap e = \kappa \ap w = \ask \ap a \ap \kappa$.
\end{proof}

The theorem lets us reason about elements of $\sheafify[\Trunc P] X$ by induction on equifoliate trees, where the inductive step has direct access to the untruncated container~$P$.


\section{Modalities in realizability toposes}
\label{sec:modalities-realizability}

Finally, we take a closer look at oracle modalities in the context of realizability.
As the topic is quite specialized, we content ourselves with brief descriptions and assume familiarity with
the notation and terminology of~\cite[Sec.~3]{Bauer:InstanceReducibility}, where a quick overview of realizability theory is available, or with the standard reference for realizability theory by Jaap van Oosten~\cite{oosten08:_realiz}.
We work with a \emph{relative realizability topos} $\RT{\AA, \AA'}$ over a partial combinatory algebra~$\AA$ with an elementary subalgebra~$\AA' \subseteq \AA$.
Our goal is to give an explicit description of modalities in~$\RT{\AA, \AA'}$.

\begin{lemma}
  \label{lem:surjective-prop-oracle-eq}%
  If $q : B \to A$ is surjective and $P : A \to \Prop$ then $\oracle{P} = \oracle{P \circ q}$.
\end{lemma}

\begin{proof}
  By \cref{prop:oracle-least-above-instance}, because $P$ and $P \circ q$ are instance reducible to each other.
\end{proof}

\Cref{lem:surjective-prop-oracle-eq} allows us to restrict attention to oracles induced by predicates on partitioned assemblies, because every object is covered by one~\cite[Prop.~3.2.7]{oosten08:_realiz}. A \defemph{partitioned assembly} $\asm{X} = (X, \rho_X)$ is given by a set~$X$ and a map $\rho_X : X \to \AA$ which assigns to each $x \in X$ its realizer~$\rho_X \ap x$. Note that different elements may share the same realizer. A predicate on a partitioned assembly~$\asm{X}$ is given by a map $P : X \to \pow{\AA}$.

A straightforward calculation shows that the oracle modality $\oracle{P}$ is represented by a map $\Oracle{P} : \pow{\AA} \to \pow{\AA}$, whose value at $S \in \pow{\AA}$ is the least $\Oracle{P} \ap S \in \pow{\AA}$ satisfying
\begin{equation}
  \label{eq:oracle-asm}
  \textstyle
  \Oracle{P} \ap S =
  \begin{aligned}[t]
  &\set{\pair{\numeral{0}, a} \such a \in S}
  \cup {} \\
  &\textstyle\bigcup_{x \in X} \set{\pair{\numeral{1}, \pair{\rho_X \ap x, c}} \such
    c \in \AA \land
    \all{d \in P \ap x} (\defined{(c \app d)} \land c \app d \in \Oracle{P} \ap S)}.
  \end{aligned}
\end{equation}
An element of $\Oracle{P} \ap S$ encodes a well-founded tree. A leaf takes the form $\pair{\numeral{0}, a}$ where $a \in S$. A node is a triple $\pair{\numeral{1}, \pair{\rho_X \ap x, b}}$ for some $x \in X$, with $\rho_X \ap x$ encoding information about~$x$, and~$b$ computing the $P \ap x$-branching subtrees.

We may rephrase~\eqref{eq:oracle-asm} in terms of Weihrauch reducibility~\cite{brattka11:_weihr}.
An \defemph{extended Weihrauch predicate}~\cite[Def.~3.7]{Bauer:InstanceReducibility} is a map $f : \AA \to \pow{(\pow{\AA})}$, and its \defemph{support} is $\support{f} \defeq \set{r \in \AA \such f \ap r \neq \emptyset}$.
We say that~$f$ is \defemph{Weihrauch reducible} to $g : \AA \to \pow{(\pow{\AA})}$, written $f \wleq g$,
when there exist $\ell_1, \ell_2 \in \AA'$ such that for all $r \in \support{f}$:
\begin{itemize}
\item $\defined{\ell_1 \app r}$ and $\ell_1 \app r \in \support{g}$,
\item for every $\theta \in f \ap r$ there is $\xi \in g \ap (\ell_1 \app r)$ such that $\ell_2 \app r \rz \xi \lthen \theta$.
\end{itemize}
(Recall from~\cite[Sec.~3.2]{Bauer:InstanceReducibility} that $\ell_2 \app r \rz \xi \lthen \theta$ means: if $s \in \xi$ then $\defined{(\ell_2 \app r \app s)}$ and $\ell_2 \app r \app s \in \theta$.)

The (large) preorder of instance reducibilities in~$\RT{\AA, \AA'}$ is equivalent to the (small) preorder~$\wleq$ on extended Weihrauch predicates~\cite[Prop.~3.8]{Bauer:InstanceReducibility}. Therefore, no generality is lost when
oracle modalities are expressed in terms of extended Weihrauch predicates.

The oracle modality $\Oracle{f} : \pow{\AA} \to \pow{\AA}$ induced by an extended Weihrauch predicate~$f$ maps $S \in \pow{\AA}$ to the least $\Oracle{f} \ap S$ satisfying
\begin{equation}
  \label{eq:oracle-weihrauch}
  \Oracle{f} \ap S =
  \begin{aligned}[t]
  &\set{\pair{\numeral{0}, a} \such a \in S} \cup {}  \\
  &\textstyle%
  \bigcup_{b \in \AA}
  \bigcup_{\theta \in f \ap b}
  \set{\pair{\numeral{1}, b, c} \such
       c \in \AA \land
       \all{d \in \theta} (\defined{(c \app d)} \land c \app d \in \Oracle{f} \ap S)}.
  \end{aligned}
\end{equation}
One can check that the above-mentioned equivalence~\cite[Prop.~3.8]{Bauer:InstanceReducibility} transforms~\eqref{eq:oracle-asm} to~\eqref{eq:oracle-weihrauch}. We see again that the elements of $\Oracle{f} \ap S$ encode well-founded trees, except that tree branching is organized differently from~\eqref{eq:oracle-asm}, namely, by a realizer $b \in \AA$ and $\theta \in f \ap b$. Note that the realizer $\pair{\numeral{1}, b, c}$ carries no information about~$\theta$.

\begin{theorem}
  In a relative realizability topos $\RT{\AA, \AA'}$, every Lawvere-Tierney topology is of the form~\eqref{eq:oracle-asm}, equivalently~\eqref{eq:oracle-weihrauch}.
\end{theorem}

\begin{proof}
  Combine \cref{cor:every-modality-oracle} and \cref{lem:surjective-prop-oracle-eq}.
\end{proof}

The formula~\eqref{eq:oracle-weihrauch} reveals a close connection with Takayuki Kihara's game-theoretic description~\cite{kihara23:_lawver_tiern} of modalities. An extended Weihrauch predicate~$f : \AA \to \pow{(\pow{\AA})}$ describes a game: Arthur either terminates the game with a final move $a \in S$, or plays a move $b \in \support{f}$, after which Nimue plays $\theta \in f \ap b$, and Merlin responds with $d \in \theta$.
Kihara's winning strategies for Arthur and Nimue correspond to well-foundedness of trees encoded by $\Oracle{f} \ap S$.


\section{Related and further work}
\label{sec:related-work}

Sheaves in the context of type theory have been studied by a number of authors.
Erik Palmgren~\cite{palmgren97:_const_sheaf_seman} and Ieke Moerdijk and Erik Palmgren~\cite{moerdijk02:_type} gave an early predicative account of sheaves.
In the context of homotopy type theory, Kevin Quirin and Nicolas Tabareau~\cite{quirin16:_lawver_tiern_homot_type_theor} formalized sheaves for general modalities,
Egbert Rijke, Mike Shulman, and Bas Spitters~\cite{rijke-shulman-spitters-2020} gave a general account of modalities,
and Thierry Coquand, Fabian Ruch, and Christian Sattler~\cite{coquand21:_const} studied constructive sheaf models of type theory.

Closer to our work are studies of sheaves that focus on computation and computability.
Martín Escardó and Chuangjie Xu~\cite{xu13:_const_model_unifor_contin} used sheaves on the Cantor space in type theory to show that every System~T definable functional is continuous.
We were inspired by the work of Andrew Swan~\cite{swan24:_oracl} on oracle modalities in cubical assemblies.
Martin Baillon, Assia Mahboubi, and Pierre-Marie Pédrot~\cite{baillon22:_garden_pythia_model_contin_depen_settin} extended the continuity results of Escardó and Xu to dependent functionals, with oracle query-answer trees featuring prominently.
In recent work~\cite{baillon26:_cantor_space}, the same authors describe sheafification as a quotient-inductive type parameterized by a \emph{logical operating system} $(I, O)$, which is essentially one of our propositional containers. Their quotient-inductive type imposes a single equation, in our notation
$\ask \ap a \ap (\lam{u \of P \ap a} v) = v$,
which under our working assumptions is equivalent to the two equations in \cref{def:sheafification-qit}.
%
%
%
%
%
%
They also use Escardó's \emph{dialogue trees}, which correspond to our unrestricted $\Tree[P]$ of \cref{sec:oracle-modal-pred} for a proof-relevant~$P$, and apply the construction to build a syntactic model of Martin-Löf type theory with a Cohen real, deriving uniform continuity of all definable Cantor-space functionals.

\section{Conclusion}
\label{sec:conclusion}

In the present work the computational viewpoint of sheaves as oracle computations is central.
Our description of a modality in terms of oracle queries and answers is both concrete
and general.
It reveals a connection between modalities and containers that will hopefully turn out to be useful in the future.
Sheafification as a quotient-inductive type is too extensional to give a faithful account of oracle computations. We have tried to improve on it by identifying the notion of equifoliate trees. These are not quotiented, but rather restricted in what answers they provide, and they are parameterized by general, non-propositional, containers.
It remains to be seen whether they beget a satisfactory theory of what one might call ``intensional sheaves''.

\paragraph*{Formalization}
We formalized \cref{sec:background,sec:oracle-modalities,sec:sheaves} in Rocq with the help of Claude Opus~4.7. The development is available at \cite{ahman26:_sheav}.




\bibliographystyle{plainurl}
\bibliography{references}

\end{document}